\RequirePackage{fix-cm}
\documentclass[headsepline=true]{scrartcl}
\usepackage[british]{babel}
\usepackage{amsmath}
\usepackage{amsthm, 	amssymb, amsfonts}
\usepackage{fancyhdr}
\usepackage[top=1.2in,bottom=1.2in,left=1in,right=1in]{geometry}
\usepackage{mathrsfs}
\usepackage{sectsty}
\usepackage[pdfborder={0 0 0}]{hyperref}
\usepackage{color}
\usepackage{enumerate}
     
\title{When is the Bloch--Okounkov $q$-bracket~modular?}
\date{\today}
\author{Jan-Willem M. van Ittersum%
\thanks{\emph{Email}: \href{mailto:j.w.m.vanittersum@uu.nl}{j.w.m.vanittersum@uu.nl}, \newline
Mathematisch Instituut, Universiteit Utrecht, Postbus 80.010, 3508 TA Utrecht, The Netherlands, \newline
Max-Planck-Institut f\"ur Mathematik, Vivatsgasse 7, 53111 Bonn, Germany
}}


\newcommand{\n}{\mathbb{N}}
\newcommand{\z}{\mathbb{Z}}
\newcommand{\q}{\mathbb{Q}}

\renewcommand{\c}{\mathbb{C}}
\newcommand{\pdv}[2]{\frac{\partial #1}{\partial #2}}

\newcommand{\tLambda}{\tilde\Lambda}
\newcommand{\tH}{\tilde{\mathcal{H}}}
\newcommand{\tR}{\tilde{\mathcal{R}}}

\newcommand{\sltwo}{\mathfrak{sl}_2}
\newcommand{\projection}{\mathrm{pr}}
\newcommand{\Kelvin}{K}

\renewcommand{\vec}{\underline}
\renewcommand{\phi}{\varphi}

\DeclareMathOperator{\sgn}{sgn}

\theoremstyle{plain} 
\newtheorem{theorem}{Theorem}
\newtheorem{lemma}[theorem]{Lemma} 
\newtheorem{corollary}[theorem]{Corollary} 
\newtheorem{proposition}[theorem]{Proposition}

\theoremstyle{definition} 
\newtheorem{definition}[theorem]{Definition}
\newtheorem{example}[theorem]{Example}

\theoremstyle{remark}
\newtheorem*{remarkx}{Remark}
\newenvironment{remark}
  {\pushQED{\qed}\remarkx}
  {\popQED\endremarkx}

\begin{document} 
\maketitle
\begin{abstract}
We obtain a condition describing when the quasimodular forms given by the Bloch--Okounkov theorem as $q$-brackets of certain functions on partitions are actually modular. This condition involves the kernel of an operator $\Delta$. We describe an explicit basis for this kernel, which is very similar to the space of classical harmonic polynomials. 
%
\end{abstract}


\section{Introduction}
Given a family of quasimodular forms, the question which of its members are modular often has an interesting answer. For example, consider the family of theta series 
$$\theta_P(\tau) = \sum_{\vec{x}\in \z^r} P(\vec{x})q^{x_1^2+\ldots+x_r^2} \quad \quad \quad (q=e^{2\pi i \tau})$$
given by all homogeneous polynomials $P\in \z[x_1,\ldots, x_r]$. The quasimodular form $\theta_P$ is modular if and only if $P$ is harmonic (i.e., $P\in \ker\sum_{i=1}^r \pdv{^2}{x_i^2}$) \cite{Sch39}%
\footnote{As quasimodular forms were not yet defined, Schoeneberg only showed that $\theta_P$ is modular if $P$ is harmonic. However, for every polynomial $P$ it follows that $\theta_P$ is quasimodular by decomposing $P$ as in Formula~(\ref{eq:mainharmpol}) below.}. 
Also, for every two modular forms $f,g$ one can consider the linear combination of products of derivatives of $f$ and $g$ given by
$$\sum_{r=0}^n a_r f^{(r)} g^{(n-r)} \quad \quad (a_r\in \c).$$
This linear combination is a quasimodular form which is modular precisely if it is a multiple of the Rankin-Cohen bracket $[f,g]_n$ \cite{Ran56,Coh75}. 
In this paper, we provide a condition to decide which member of the family of quasimodular forms provided by the Bloch--Okounkov theorem is modular. Let $\mathscr{P}$ denote the set of all partitions of integers and $|\lambda|$ denotes the integer that $\lambda$ is a partition of. Given a function $f:\mathscr{P}\to \q$, define the $q$-bracket of $f$ by
$$\langle f \rangle_q := \frac{\sum_{\lambda \in \mathscr{P}} f(\lambda) q^{|\lambda|}}{\sum_{\lambda \in \mathscr{P}} q^{|\lambda|}}.$$
The celebrated Bloch--Okounkov theorem states that for a certain family of functions $f:\mathscr{P}\to \q$ (called shifted symmetric polynomials and defined in Section~\ref{sec:ssp}) the $q$-brackets $\langle f \rangle_q$ are the $q$-expansions  of quasimodular forms \cite{BO00}.

Besides being a wonderful result, the Bloch--Okounkov theorem has many application in enumerative geometry. For example, a special case of the Bloch--Okounkov theorem was discovered by Dijkgraaf and provided with a mathematically rigorous proof by Kaneko and Zagier, implying that the generating series of simple Hurwitz numbers over a torus are quasimodular \cite{Dij95,KZ95}. Also, in the computation of asymptotics of geometrical invariants, such as volumes of moduli spaces of holomorphic differentials and Siegel-Veech constants the Bloch--Okounkov theorem is applied \cite{EO01,CMZ16}. 

Zagier gave a surprisingly short and elementary proof of the Bloch--Okounkov theorem \cite{Zag16}. A corollary of his work, which we discuss in Section~\ref{sec:3}, is the following proposition:
\begin{proposition}\label{prop:sltwobo1}
There exists actions of the Lie algebra $\sltwo$ on both the algebra of shifted symmetric polynomials $\Lambda^*$ and the algebra of quasimodular forms $\widetilde{M}$ such that the $q$-bracket $\langle \cdot \rangle_q: \Lambda^*\to \widetilde{M}$ is $\sltwo$-equivariant.  
\end{proposition}

The answer to the question in the title is provided by one of the operators $\Delta$ which defines this $\sltwo$-action on $\Lambda^*$. Namely, letting $\mathcal{H}=\ker \Delta|_{\Lambda^*}$, we prove the following theorem:

\begin{theorem}\label{thm:harmonicmodular}
Let $f\in \Lambda^*$. Then $\langle f \rangle_q$ is modular if and only if $f=h+k$ with $h\in \mathcal{H}$ and $k\in \ker\langle \cdot \rangle_q$.  
\end{theorem}

The last section of this article is devoted to describing the graded algebra $\mathcal{H}$. We call $\mathcal{H}$ the space of \emph{shifted symmetric \emph{harmonic} polynomials}, as the description of this space turns out to be very similar to the space of classical harmonic polynomials. Let $\mathcal{P}_d$ be the space of polynomials of degree $d$ in $m\geq 3$ variables $x_1,\ldots,x_m$, let $||x||^2=\sum_{i}x_i^2$ and recall that the space $\mathscr{H}_d$ of degree $d$ harmonic polynomials is given by $\ker\sum_{i=1}^r \pdv{^2}{x_i^2}$. The main theorem of harmonic polynomials states that every polynomial $P\in \mathcal{P}_d$ can uniquely be written in the form
\begin{equation}\label{eq:mainharmpol}
P=h_0+||x||^2h_1+\ldots+||x||^{2d'}h_{d'}
\end{equation}
with $h_i\in \mathscr{H}_{d-2i}$ and $d'=\lfloor d/2\rfloor$. Define $\Kelvin $, the Kelvin transform, and $D^\alpha$ for $\alpha$ an $m$-tuple of non-negative integers by
\[f(x)\mapsto ||x||^{2-m}f\left(\frac{x}{||x||^2}\right) \quad \text{and} \quad D^\alpha = \prod_i \pdv{^\alpha_i}{x_i^{\alpha_i}}.\]  An explicit basis for $\mathscr{H}_d$ is given by 
$$\{\Kelvin D^\alpha \Kelvin (1) \mid \alpha \in \z_{\geq 0}^m, \textstyle\sum_i \alpha_i=d, \alpha_1\leq 1\},$$
see for example \cite{ABR01}. We prove the following analogous results for the space of shifted symmetric polynomials:
\begin{theorem}\label{thm:exp}
For every $f\in \Lambda^*_n$ there exists unique $h_{i}\in \mathcal{H}_{n-2i}$ \textup{(}$i=0,1,\ldots, n'$ and $n'={\lfloor \tfrac{n}{2}\rfloor}$\textup{)} such that
\[f= h_0+Q_2h_1+\ldots+Q_2^{n'}h_{n'}, \]
where $Q_2$ is an element of $\Lambda^*_2$ given by $Q_2(\lambda)=|\lambda|-\frac{1}{24}$.
\end{theorem}

\begin{theorem}\label{thm:B}
The set
$$\{\projection\, \Kelvin\, \Delta _\lambda\, \Kelvin (1) \mid \lambda \in \mathscr{P}(n), \text{all parts are } \geq 3\}$$
is a vector space basis of $\mathcal{H}_n$, where $\projection, \Kelvin $ and $\Delta_\lambda$ are defined by {\upshape(\ref{eq:def_pi})}, Definition~\ref{def:K} respectively Definition~\ref{def:Dl}.
\end{theorem}
The action of $\sltwo$ given by Proposition~\ref{prop:sltwobo1} makes $\Lambda^*$ into an infinite-dimensional $\sltwo$-representation for which the elements of $\mathcal{H}$ are the lowest weight vectors. Theorem~\ref{thm:exp} is equivalent to the statement that $\Lambda^*$ is a direct sum of the (not necessarily irreducible) lowest weight modules 
\[\displaystyle V_n=\bigoplus_{m=0}^\infty Q_2^m \mathcal{H}_n \quad \quad (n\in \z).\]

\section{Shifted symmetric polynomials}\label{sec:ssp}
Shifted symmetric polynomials were introduced by Okounkov and Olshanski as the following analogue of symmetric polynomials \cite{OO97}. Let $\Lambda^*(m)$ be the space of rational polynomials in $m$ variables $x_1,\ldots, x_m$ which are \emph{shifted symmetric}, i.e. invariant under the action of all $\sigma\in\mathfrak{S}_m$ given by $x_i\mapsto x_{\sigma(i)}+i-\sigma(i)$ (or more symmetrically $x_i-i\mapsto x_{\sigma(i)}-\sigma(i)$). Note that $\Lambda^*(m)$ is filtered by the degree of the polynomials. We have forgetful maps $\Lambda^*(m)\to \Lambda^*({m-1})$ given by $x_m\mapsto 0$, so that we can define the space of shifted symmetric polynomials $\Lambda^*$ as $\displaystyle\varprojlim_m \Lambda^*(m)$ in the category of filtered algebras. Considering a partition $\lambda$ as a non-increasing sequence $(\lambda_1,\lambda_2,\ldots)$ of non-negative integers $\lambda_i$, we can interpret $\Lambda^*$ as being a subspace of all functions $\mathscr{P}\to \q$. 

One can find a concrete basis for this abstractly defined space by considering the generating series
\begin{equation}\label{eq:gs} w_\lambda(T):=\sum_{i=1}^\infty T^{\lambda_i-i+\tfrac{1}{2}} \in T^{1/2}\z[T][[T^{-1}]] \end{equation}
for every $\lambda\in\mathscr{P}$ (the constant $\tfrac{1}{2}$ turns out to be convenient for defining a grading on $\Lambda^*$). As $w_\lambda(T)$ converges for $T>1$ and equals 
$$\frac{1}{T^{1/2}-T^{-1/2}} + \sum_{i=1}^{\ell(\lambda)}\left( T^{\lambda_i-i+\tfrac{1}{2}}-T^{-i+\tfrac{1}{2}}\right)$$
one can define shifted symmetric polynomials $Q_i(\lambda)$ for $i\geq 0$ by
\begin{equation}\label{eq:def_qk}\sum_{i=0}^\infty Q_i(\lambda) z^{i-1} := w_\lambda(e^z) \quad \quad (0<|z|<2\pi).\end{equation}
The first few shifted symmetric polynomials $Q_i$ are given by
$$Q_0(\lambda)=1,\quad Q_1(\lambda)=0,\quad Q_2(\lambda)=|\lambda|-\tfrac{1}{24}.$$
The $Q_i$ freely generate the algebra of shifted symmetric polynomials, i.e. $\Lambda^*=\q[Q_2,Q_3,\ldots]$. It is believed that $\Lambda^*$ is maximal in the sense that for all $Q:\mathscr{P}\to \q$ with $Q\not \in \Lambda^*$ it holds that $\langle \Lambda^*[Q]\rangle_q \not \subseteq \widetilde{M}$.

\begin{remark} The space $\Lambda^*$ can equally well be defined in terms of the Frobenius coordinates. Given a partition with Frobenius coordinates $(a_1,\ldots, a_r,b_1,\ldots, b_r)$, where $a_i$ and $b_i$ are the arm- and leg-lenghts of the cells on the main diagonal, let 
$$C_\lambda=\{-b_1-\tfrac{1}{2},\ldots, -b_r-\tfrac{1}{2},a_r+\tfrac{1}{2},\ldots,a_1+\tfrac{1}{2} \}.$$
Then
$$Q_k(\lambda) =\beta_k + \frac{1}{(k-1)!}\sum_{c\in C_{\lambda}} \sgn(c)c^{k-1},$$
where $\beta_k$ is the constant given by
\[\sum_{k\geq 0}\beta_k z^{k-1} = \frac{1}{2\sinh(z/2)} = w_\emptyset(e^z). \qedhere\]
\end{remark}

We extend $\Lambda^*$ to an algebra where $Q_1\not \equiv 0$. Observe that a non-increasing sequence $(\lambda_1,\lambda_2,\ldots)$ of integers corresponds to a partition precisely if it converges to $0$. If, however, it converges to an integer $n$, Equations~(\ref{eq:gs}) and (\ref{eq:def_qk}) still define $Q_k(\lambda)$. In fact, in this case 
$$Q_k(\lambda) = (e^{n\boldsymbol{\partial}})Q_k(\lambda-n)$$
by \cite[Proposition 1]{Zag16} where $\boldsymbol{\partial} Q_0=0$, $\boldsymbol{\partial} Q_k = Q_{k-1}$ for $k\geq 1$ and $\lambda-n=(\lambda_1-n,\lambda_2-n,\ldots)$ corresponds to a partition (i.e. converges to $0$). In particular $Q_1(\lambda)=n$ equals the number the sequence $\lambda$ converges to. We now define the Bloch--Okounkov ring $\mathcal{R}$ to be $\Lambda^*[Q_1]$, considered as a subspace of all functions from non-increasing eventually constant sequences of integers to $\q$. It is convenient to work with $\mathcal{R}$ instead of $\Lambda^*$ to define the differential operators $\Delta$ and more generally $\Delta_\lambda$ later. Both on $\Lambda^*$ and $\mathcal{R}$ we define a weight grading by assigning to $Q_i$ weight $i$. Denote the projection map by 
\begin{align}\label{eq:def_pi} \projection:\mathcal{R}\to \Lambda^*.\end{align} 
We extend $\langle \cdot \rangle_q$ to $\mathcal{R}$. 

The operator $E=\sum_{m=0}^\infty Q_m \pdv{}{Q_{m}}$ on $\mathcal{R}$ multiplies an element of $\mathcal{R}$ by its weight. Moreover, we consider the differential operators
$$\boldsymbol{\partial} = \sum_{m=0}^\infty Q_m \pdv{}{Q_{m+1}} \quad \text{and} \quad \mathscr{D} = \sum_{k,\ell\geq 0} \binom{k+\ell}{k} Q_{k+\ell} \frac{\partial^2}{\partial Q_{k+1} \partial Q_{\ell+1}}. $$
Let $\Delta =\tfrac{1}{2}(\mathscr{D}-\boldsymbol{\partial}^2)$, i.e.
$$2\Delta  = \sum_{k,\ell\geq 0}\left(\binom{k+\ell}{k}Q_{k+\ell}-Q_kQ_\ell\right) \frac{\partial^2}{\partial Q_{k+1} \partial Q_{\ell+1}}-\sum_{k\geq 0} Q_k \pdv{}{Q_{k+2}}.$$
In the following (antisymmetric) table the entry in the row of operator $A$ and column of operator $B$ denotes the commutator $[A,B]$, for proofs see \cite[Lemma 3]{Zag16}.
\setlength\arraycolsep{8pt}
\renewcommand{\arraystretch}{1.2}
$$\begin{array}{c | c c c c c}\label{tbl:comm}
&\Delta & \boldsymbol{\partial} & E& Q_1 & Q_2 \\ \hline
\Delta & 0 & 0 & 2\Delta & 0 & E-Q_1\boldsymbol{\partial}-\tfrac{1}{2} \\
\boldsymbol{\partial} & 0 & 0 & \boldsymbol{\partial}& 1 &Q_1 \\
E & -2\Delta & -\boldsymbol{\partial}  & 0 & Q_1 & 2Q_2 \\
Q_1 & 0 & -1 & -Q_1 & 0 & 0 \\
Q_2 & -E+Q_1\boldsymbol{\partial}+\tfrac{1}{2} & 0 & -2 Q_2 & 0 & 0
\end{array}$$

\begin{definition} A triple $(X,Y,H)$ of operators is called an \emph{$\sltwo$-triple} if
$$[H,X]=2X, \quad [H,Y]=-2Y, \quad [Y,X]=H.$$
\end{definition}

Let $\hat{Q}_2:=Q_2-\tfrac{1}{2}Q_1^2$  and  $\hat{E}:=E-Q_1\boldsymbol{\partial}-\frac{1}{2}.$ The following result follows by a direct computation using above table:
\begin{proposition}
The operators $(\hat{Q}_2,\Delta,\hat{E})$ form an $\sltwo$-triple. \hfill \mbox{\qed}
\end{proposition}

For later reference we compute $[\Delta,Q_2^n]$. This could be done inductively by noting that $[\Delta,Q_2^n] = Q_2^{n-1}[\Delta,Q_2]+[\Delta,Q_2^{n-1}]Q_2$ and using the commutation relations in above table. The proof below is a direct computation from the definition of $\Delta$. 

\begin{lemma}\label{lem:DeltaQ2}
For all $n\in \n$ the following relation holds
$$[\Delta,Q_2^n]=-\frac{n(n-1)}{2}Q_1^2Q_2^{n-2}-nQ_1Q_2^{n-1}\boldsymbol{\partial}+nQ_2^{n-1}(E+n-\tfrac{3}{2}).$$
\end{lemma}
\begin{proof}
Let $f\in \q[Q_1,Q_2]$, $g\in \mathscr{R}$ and $a\in \n$. Then 
\begin{align}\label{eq:Dfg}\Delta (fg) &= \Delta (f) g + \pdv{f}{Q_2}(Eg-Q_1 \boldsymbol{\partial} g)+f\Delta (g), \\
\label{eq:Dq_2^a}\Delta (Q_2^n)&=n(n-\tfrac{3}{2})Q_2^{n-1}-n(n-1)Q_2^{n-2}Q_1^2.
\end{align}
By (\ref{eq:Dfg}) and (\ref{eq:Dq_2^a}) we find
\[\Delta (Q_2^ng)=
(n(n-\tfrac{3}{2})Q_2^{n-1}-\frac{n(n-1)}{2}Q_1^2Q_2^{n-2})g+nQ_2^{n-1}(Eg-Q_1 \boldsymbol{\partial} g)+Q_2^{n}\Delta (g). \qedhere \]
\end{proof}

\section{An $\sltwo$-equivariant mapping}\label{sec:3}
The space of quasimodular forms for $\mathrm{SL}_2(\z)$ is given by $\widetilde{M} =\q[P,Q,R]$, where $P,Q$ and $R$ are the Eisenstein series of weight $2,4$ and $6$ respectively (in Ramanujan's notation). We let $\widetilde{M}^{(\leq p)}_k$ be the space of quasimodular forms of weight $k$ and depth $\leq p$ (the depth of a quasimodular form written as a polynomial in $P,Q$ and $R$ is the degree of this polynomial in $P$). See \cite[Section 5.3]{Zag08} or \cite[Section 2]{Zag16} for an introduction into quasimodular forms. 

The space of quasimodular forms is closed under differentiation, more precisely the operators 
$D=q\frac{d}{dq}$, $\mathfrak{d}=12\pdv{}{P}$ and the weight operator $W$ given by $Wf=kf$ for $f\in \widetilde{M}_k$ preserve $\widetilde{M}$ and form an $\mathfrak{sl}_2$-triple. In order to compute the action of $D$ in terms of the generators $P,Q$ and $R$ one uses the Ramanujan identities
$$D(P)=\frac{P^2-Q}{12}, \quad D(Q)=\frac{PQ-R}{3}, \quad D(R)= \frac{PR-Q^2}{2}.$$
In the context of the Bloch--Okounkov theorem it is more natural to work with $\hat{D} := D - \frac{P}{24}$, as for all $f\in \Lambda^*$ one has $\langle Q_2 f\rangle_q = \hat{D} \langle f \rangle_q$. Moreover, $\hat{D}$ has the property that it increases the depth of a quasimodular form by $1$, in contrast to $D$ for which $D(1)=0$ does not have depth $1$: 

\begin{lemma}\label{lem:depth}
Let $f\in \widetilde{M}$ be of depth $r$. Then $\hat{D}f$ is of depth $r+1$. 
\end{lemma}
\begin{proof}
Consider a monomial $P^aQ^bR^c$ with $a,b,c\in \z_{\geq 0}$. By the Ramanujan identities we find
$$D(P^aQ^bR^c) = \left(\frac{a}{12}+\frac{b}{3}+\frac{c}{2}\right)P^{a+1}Q^bR^c + O(P^{a}),$$
where $O(P^{a})$ denotes a quasimodular form of depth at most $a$. The lemma follows by noting that $\frac{a}{12}+\frac{b}{3}+\frac{c}{2}-\frac{1}{24}$ is non-zero for $a,b,c\in \z$.
\end{proof}

Moreover, letting $\hat{W}=W-\frac{1}{2}$, the triple ($\hat{D}, \mathfrak{d}, \hat{W})$ forms an $\mathfrak{sl}_2$-triple as well. With respect to these operators the $q$-bracket becomes $\sltwo$-equivariant. The following proposition is a detailed version of Proposition~\ref{prop:sltwobo1}:
\begin{proposition}[The $\sltwo$-equivariant Bloch--Okounkov theorem]
The mapping $\langle \cdot \rangle_q: \mathcal{R} \to \widetilde{M}$ is $\sltwo$-equivariant with respect to the $\sltwo$-triple $(\hat{Q}_2, \Delta,\hat{E})$ on $\mathcal{R}$ and the $\sltwo$-triple $(\hat{D}, \mathfrak{d}, \hat{W})$ on $\widetilde{M}$, i.e. for all $f\in \mathcal{R}$ one has
$$\hat{D}\langle f\rangle_q = \langle \hat{Q}_2f\rangle_q,\quad \mathfrak{d}\langle f\rangle_q=\langle \Delta f\rangle_q, \quad \hat{W}\langle f\rangle_q= \langle \hat{E}f\rangle_q.$$
\end{proposition}
\begin{proof}
This follows directly from \cite[Equation (37)]{Zag16} and the fact that for all $f\in \mathcal{R}$ one has $\langle Q_1 f \rangle_q=0$. 
\end{proof}

\section{Describing the space of shifted symmetric harmonic polynomials}
In this section we study the kernel of $\Delta$. As $[\Delta,Q_1]=0$, we restrict ourselves without  loss of generality to $\Lambda^*$. Note, however, that $\Delta$ does not act on $\Lambda^*$ as for example $\Delta(Q_3)=-\tfrac{1}{2}Q_1$. However, $\projection\Delta$ does act on $\Lambda^*$. 

\begin{definition}
Let $$\mathcal{H} = \{f \in \Lambda^* \mid \Delta f\in Q_1\mathcal{R}\}=\ker \projection\Delta,$$
be the space of \textit{shifted symmetric harmonic} polynomials. 
\end{definition}

\begin{proposition}\label{prop:q2multiple}
If $f\in Q_2\Lambda^*$ is non-zero, then $f \not \in \mathcal{H}$. 
\end{proposition}
\begin{proof}
Write $f=Q_2^n f'$ with $f'\in \Lambda^*$ and $f'\not \in Q_2\Lambda^*$. Then
$$\projection\Delta(f)=Q_2^{n-1}(n(n+k-\tfrac{3}{2})f'+Q_2\projection\Delta f')$$
by Lemma~\ref{lem:DeltaQ2}.  As $f'$ is not divisible by $Q_2$, it follows that $\projection \Delta(f) =0$ precisely if $f'=0$. 
\end{proof}

\begin{proposition}\label{prop:oplus} For all $n\in \z$ one has
$$\Lambda_n^* = \mathcal{H}_n \oplus Q_2 \Lambda_{n-2}^*.$$
\end{proposition}
\begin{proof}
For uniqueness, suppose $f=Q_2g+h$ and $f=Q_2g'+h'$ with $g,g'\in \Lambda_{n-2}^*$ and $h,h' \in \mathcal{H}_n$. Then, $Q_2(g-g')=h'-h\in \mathcal{H}$. By Proposition~\ref{prop:q2multiple} we find $g=g'$ and hence $h=h'$. 

Now, define the linear map $T:\Lambda_n^*\to \Lambda_n^*$ by
$f\mapsto\projection\Delta (Q_2f).$
 By Proposition~\ref{prop:q2multiple} we find that $T$ is injective, which by finite dimensionality of $\Lambda_n^*$ implies that $T$ is surjective. Hence, given $f\in \Lambda_{n}^*$ let $g\in \Lambda_{n-2}^*$ be such that $T(g)=\projection\Delta (f) \in \Lambda_{n-2}^*$. Let $h=f-Q_2g$. As $f=Q_2g+h$, it suffices to show that $h\in \mathcal{H}$. That holds true because $\projection\Delta (h)=\projection\Delta (f)-\projection\Delta (Q_2g)=0$.
\end{proof}
Proposition~\ref{prop:oplus} implies Theorem~\ref{thm:exp} and the following corollary. Denote by $p(n)$ the number of partitions of $n$. 
\begin{corollary}\label{cor:dim} The dimension of $\mathcal{H}_n$ equals the number of partitions of $n$ in at least $3$ parts, i.e.
\begin{align*}
\dim \mathcal{H}_n = p(n)-p(n-1)-p(n-2)+p(n-3). 
\end{align*} 
\end{corollary}
\begin{proof}
Observe that $\dim \Lambda_n^*$ equals the number of partitions of $n$ in at least $2$ parts. Hence, $\dim \Lambda_n^*=p(n)-p(n-1)$ and the Corollary follows from Proposition~\ref{prop:oplus}.
\end{proof}

\begin{proof}[Proof of Theorem~\ref{thm:harmonicmodular}]
If $\langle f \rangle_q$ is modular, then $\langle \Delta f \rangle_q=\mathfrak{d}\langle f \rangle_q =0$. Write $f=\sum_{r=0}^{n'} Q_2^rh_{r}$ as in Theorem~\ref{thm:exp} with $n'=\lfloor \tfrac{n}{2}\rfloor$. Then by Lemma~\ref{lem:DeltaQ2} it follows that $\projection \Delta f = \sum_{r=0}^{n'} r(n-r-\frac{3}{2})Q_2^{r-1}h_r.$ Hence,
\begin{align}\label{eq:expdelta}\sum_{r=1}^{n'} r(n-r-\frac{3}{2})\hat{D}^{r-1}\langle h_r\rangle_q=0.\end{align}
As $\langle h_r\rangle_q$ is modular, either it is equal to $0$ or it has depth $0$. Suppose the maximum $m$ of all $r\geq 1$ such that $\langle h_r\rangle_q$ is non-zero exists. Then, by 
Lemma~\ref{lem:depth} it follows that the left-hand side of (\ref{eq:expdelta}) has depth $m-1$, in particular is not equal to $0$.  So, $h_1,\ldots, h_{n'}\in \ker \langle \cdot \rangle_q$. Note that $f\in \ker \langle \cdot \rangle_q$ implies that $Q_2f\in \ker \langle \cdot \rangle_q$. Therefore, $k:=\sum_{r=1}^{n'} Q_2^rh_{r} \in \ker \langle \cdot \rangle_q$ and $f=h+k$ with $h=h_0$ harmonic.
 
The converse follows directly as $\mathfrak{d}\langle h+k\rangle_q = \mathfrak{d}\langle  h \rangle_q=\langle \Delta h \rangle_q = 0.$ 
\end{proof}
\begin{remark} A description of the kernel of $\langle \cdot \rangle_q$ is not known. 
\end{remark}

Another corollary of Proposition~\ref{prop:oplus} is the notion of \emph{depth} of shifted symmetric polynomials which corresponds to the depth of quasimodular forms:
\begin{definition}\label{def:depth}
The space $\Lambda_k^{*(\leq p)}$ 
of shifted symmetric polynomials of depth $\leq p$ is the space of $f\in \Lambda_k^*$ 
such that one can write $$ f = \sum_{r=0}^{p} Q_2^rh_{r}$$
with $h_r\in \mathcal{H}_{k-2r}$. 
\end{definition}

\begin{theorem}
If $f\in \Lambda_k^{*(\leq p)}$, then $\langle f \rangle_q \in \widetilde{M}_k^{(\leq p)}$. 
\end{theorem}

\begin{proof}
Expanding $f$ as in Definition~\ref{def:depth} we find
$$\langle f \rangle_q = \sum_{k=0}^{p} \langle Q_2^kh_{k} \rangle_q = \sum_{k=0}^{p} \hat{D}^k\langle h_{k}\rangle_q.$$
By Lemma~\ref{lem:depth}, we find that the depth of $\langle f\rangle_q$ is at most $p$. 
\end{proof}

Next, we set up notation to determine the basis of $\mathcal{H}$ given by Theorem~\ref{thm:B}. Let $\tR=\mathcal{R}[Q_2^{-1/2}]$ and $\tLambda=\Lambda^*[Q_2^{-1/2}]$ be the formal polynomial algebras graded by assigning to $Q_k$ weight $k$ (note that the weights are -- possibly negative -- integers). Extend $\Delta $ to $\tLambda$ and observe that $\Delta (\tLambda)\subset \tLambda$. Also extend $\mathcal{H}$ by setting
\[\tH=\{f \in \tLambda \mid \Delta f\in Q_1\tR\}=\ker\projection\Delta |_{\tLambda}.\]

\begin{definition}\label{def:K}
Define the \textit{partition-Kelvin transform} $\Kelvin :\tLambda_n\to \tLambda_{3-n}$ by
$$\Kelvin (f)=Q_2^{3/2-n}f.$$ 
\end{definition}
Note that $\Kelvin $ is an involution. Moreover, $f$ is harmonic if and only if $K(f)$ is harmonic, which follows directly from the computation
\[\Delta \Kelvin(f)=Q_2^{3/2-n}\Delta f-(\tfrac{3}{2}-n)Q_1Q_2^{\tfrac{1}{2}-n}\boldsymbol{\partial}{f}-\tfrac{1}{2}(\tfrac{3}{2}-n)(\tfrac{1}{2}-n)Q_1^2Q_2^{-\tfrac{1}{2}-n}f. \]

\begin{example}\label{ex:Kelvin}
As $\Kelvin (1)=Q_2^{3/2}$, it follows that $Q_2^{3/2}\in \tH$.
\end{example}

\begin{definition}
Given $\vec{i}\in \z_{\geq 0}^n$, let
\[
|\vec{i}|=i_1+i_2+\ldots + i_n, \quad \quad \partial_{\vec{i}} = \frac{\partial^n}{\partial Q_{i_1+1} \partial Q_{i_2+1}\cdots \partial Q_{i_n+1}}.
\]
Define the $n$th order differential operators $\mathscr{D}_n$ on $\tR$ by
$$\mathscr{D}_n = \sum_{\vec{i}\in \z_{\geq 0}^n} \binom{|\vec{i}|}{i_1,i_2,\ldots,i_n} Q_{|\vec{i}|}\partial_{\vec{i}},$$
where the coefficient is a multinomial coefficient. 
\end{definition}
This definition generalises the operators $\boldsymbol{\partial}$ and $\mathscr{D}$ to heigher weights: $\mathscr{D}_1=\boldsymbol{\partial}$, $\mathscr{D}_2=\mathscr{D}$ and $\mathscr{D}_n$ reduces the weight by $n$. 

\begin{lemma}\label{lem:Dncom}
The operators $\{\mathscr{D}_n\}_{n\in \n}$ commute pairwise. 
\end{lemma}
\begin{proof}
Set $I=|\vec{i}|$ and $J=|\vec{j}|$. Let $\vec{a}^{\hat{k}}=(a_1,\ldots,a_{k-1},a_{k+1},\ldots, a_n)$. Then
\begin{align*}
&\left[\binom{I}{i_1,i_2,\ldots,i_n} Q_{I}\frac{\partial^n}{\partial Q_{\vec{i}}},\binom{J}{j_1,j_2,\ldots,j_m} Q_{J}\frac{\partial^m}{\partial Q_{\vec{j}}}\right] \\
=& \sum_{k=1}^n \delta_{i_k,J-1} J\binom{I}{i_1,i_2,\ldots,\hat{i_k},\ldots,i_n,j_1,j_2,\ldots,j_m} Q_I\frac{\partial^{n-1}}{\partial Q_{\vec{i}^{\hat{k}}}}\frac{\partial^m}{\partial Q_{\vec{j}}} +\\
&\quad -\sum_{l=1}^m \delta_{j_l,I-1} I\binom{J}{i_1,i_2,\ldots,i_n,j_1,j_2,\ldots,\hat{j_l},\ldots,j_m} Q_J\frac{\partial^n}{\partial Q_{\vec{i}}}\frac{\partial^{m-1}}{\partial Q_{\vec{j}^{\hat{l}}}}.
\end{align*}
Hence, $[\mathscr{D}_n,\mathscr{D}_m]$ is a linear combination of terms of the form $Q_{|\vec{a}|+1} \frac{\partial^{n+m-1}}{\partial Q_{\vec{a}}},$
where $\vec{a}\in \z_{\geq 0}^{n+m-1}$. We collect all terms for different vectors $\vec{a}$ which consists of the same parts (i.e. we group all vectors $\vec{a}$ which correspond to the same partition). Then, the coefficient of such a term equals
\begin{align*}&\sum_{k=1}^n \sum_{\sigma \in S_{m+n-1}} (a_{\sigma(1)}+\ldots + a_{\sigma(m)})\binom{|\vec{a}|+1}{a_1,a_2,\ldots, a_{n+m-1}} + \\
&\quad -  \sum_{l=1}^m \sum_{\sigma \in S_{m+n-1}} (a_{\sigma(1)}+\ldots + a_{\sigma(n)})\binom{|\vec{a}|+1}{a_1,a_2,\ldots, a_{n+m-1}} \\
=&  (mn-mn) \sum_{\sigma \in S_{m+n-1}} a_{\sigma(1)}\binom{|\vec{a}|+1}{a_1,a_2,\ldots, a_{n+m-1}} =0.
\end{align*}
Hence, $[\mathscr{D}_n,\mathscr{D}_m]=0$. 
\end{proof}

It does not hold true that $[\mathscr{D}_n,Q_1]=0$ for all $n\in \n$. Therefore, we introduce the following operators:

\begin{definition}\label{def:Dl}
Let
\[\Delta _n=\sum_{i=0}^n (-1)^i \binom{n}{i} \mathscr{D}_{n-i} \boldsymbol{\partial}^i .\]
For $\lambda \in \mathscr{P}$ let
\[\Delta _\lambda = \binom{|\lambda|}{\lambda_1,\ldots,\lambda_{\ell(\lambda)}} \prod_{i=1}^\infty \Delta _{\lambda_i}.\]
(Note that $\Delta _0=\mathscr{D}_0=1$, so this is in fact a finite product.)\\
\end{definition}
\begin{remark}
By M\"obius inversion
\[\mathscr{D}_n=\sum_{i=0}^n \binom{n}{i} \Delta _{n-i}\boldsymbol{\partial}^i. \]
\end{remark}

The first three operators are given by
$$\Delta _0=1,\quad \Delta _1=0,\quad \Delta _2=\mathscr{D}-\boldsymbol{\partial}^2=2\Delta $$

\begin{proposition}\label{prop:Dlambda}
The operators $\Delta _\lambda$ satisfy the following properties: for all partitions $\lambda,\lambda'$
\begin{enumerate}[\upshape (a)] \itemsep0pt
\item\label{it:DL1} the order of $\Delta _{|\lambda|}$ is $|\lambda|;$
\item\label{it:DL2} $[\Delta _\lambda,\Delta _{\lambda'}]=0;$
\item \label{it:DL3}$[\Delta _\lambda,Q_1]=0$. 
\end{enumerate}

\end{proposition}
\begin{proof}
Property (\ref{it:DL1}) follows by construction and (\ref{it:DL2}) is a direct consequence of Lemma~\ref{lem:Dncom}. For property (\ref{it:DL3}), let $f\in \tLambda$ be given. Then
\begin{align*}
\Delta _n(Q_1 f) &=\sum_{i=0}^n (-1)^i \binom{n}{i} \mathscr{D}_{n-i} \boldsymbol{\partial}^i(Q_1 f) \\
&=\sum_{i=0}^n (-1)^i \binom{n}{i} \left((n-i)\mathscr{D}_{n-i-1}\boldsymbol{\partial}^{i} f + Q_1\mathscr{D}_{n-i}\boldsymbol{\partial}^{i} f + i \mathscr{D}_{n-i}\boldsymbol{\partial}^{i-1} f\right) \\
&= Q_1 \Delta _n(f)+\sum_{i=0}^n (-1)^i \binom{n}{i} \left((n-i)\mathscr{D}_{n-i-1}\boldsymbol{\partial}^{i} f + i\mathscr{D}_{n-i}\boldsymbol{\partial}^{i-1} f\right).
\end{align*}
Observe that by the identity
$$(n-i)\binom{n}{i}=i\binom{n}{i+1}$$
the sum in the last line is a telescoping sum, equal to zero. Hence $\Delta _n(Q_1f)=Q_1\Delta _n(f)$ as desired. 
\end{proof}

In particular, above proposition yields $[\Delta _\lambda, \Delta ]=0$ and $[\Delta _\lambda,\projection]=0$. 

Denote by $(x)_{n}$ the \textit{falling factorial power}
$(x)_n = \prod_{i=0}^{n-1} (x-i)$
and for $\lambda\in\mathscr{P}_n$ define 
$Q_\lambda=\prod_{i=1}^\infty Q_{\lambda_i}.$
Let $$h_\lambda=\projection \Kelvin \Delta_\lambda \Kelvin (1).$$ Observe that $h_\lambda$ is harmonic, as $\projection\Delta$ commutes with $\projection$ and $\Delta_\lambda$.

\begin{proposition}\label{prop:Dlambda}
For all $\lambda\in \mathscr{P}_n$ there exists an $f\in \Lambda^*_{n-2}$ such that
$$h_\lambda = (\tfrac{3}{2})_n n! Q_\lambda + Q_2 f.$$
\end{proposition}
\begin{proof}
Note that the left hand side is an element of $\Lambda^*$ of which the monomials divisible by $Q_2^i$ correspond precisely to terms in $\Delta_\lambda$ involving precisely $n-i$ derivatives of $\Kelvin (1)$ to $Q_2$. Hence, as $\Delta_\lambda$ has order $n$ all terms not divisible by $Q_2$ correspond to terms in $\Delta_\lambda$ which equal $\pdv{}{Q_2^{n-i}}$ up to a coefficient. There is only one such term in $\Delta_{\lambda}$ with coefficient $\binom{|\lambda|}{\lambda_1,\ldots,\lambda_r}\lambda_1!\ldots\lambda_r!Q_{\lambda}$. 
\end{proof}

For $f\in \mathcal{R}$, we let $f^\vee$ be the operator where every occurence of $Q_i$ in $f$ is replaced by $\Delta _i$. We get the following unusual identity:

\begin{corollary} If $h\in \mathcal{H}_n$, then
\begin{align} \label{eq:unusual} h=\frac{\projection \Kelvin  h^\vee \Kelvin (1)}{n!(\tfrac{3}{2})_n}. \end{align}
\end{corollary}
\begin{proof}
By Proposition~\ref{prop:Dlambda} we know that the statement holds true up to adding $Q_2 f$ on the right-hand side for some $f\in \Lambda^*_{n-2}$. However, as both sides of (\ref{eq:unusual}) are harmonic and the shifted symmetric polynomial $Q_2f$ is harmonic precisely if $f=0$ by Proposition~\ref{prop:q2multiple}, it follows that $f=0$ and (\ref{eq:unusual}) holds true. 
\end{proof}

\begin{proof}[Proof of Theorem~\ref{thm:B}]
Let $\mathcal{B}_n=\{h_\lambda \mid \lambda \in \mathscr{P}_n \text{ all parts are } \geq 3\}$. First of all, observe that by Corollary~\ref{cor:dim} the number of elements in $\mathcal{B}_n$ is precisely the dimension of $\mathcal{H}_n$. Moreover, the weight of an element in $\mathcal{B}_n$ equals $|\lambda|=n$. By Proposition~\ref{prop:Dlambda} it follows that the elements of $\mathcal{B}_n$ are linearly independent harmonic shifted symmetric polynomials. 
\end{proof}

\appendix
\sectionfont{\large}
\section{Tables of shifted symmetric harmonic polynomials up to weight $10$}\small
We list all harmonic polynomials $h_{\lambda}$ of even weight at most $10$. The corresponding $q$-brackets $\langle h_\lambda\rangle_q$ are computed by the algorithm prescriped by Zagier \cite{Zag16} using SageMath \cite{sagemath}. 
\[  \arraycolsep=5pt\def\arraystretch{1.5}
\begin{array}{l l l}
\lambda & h_{\lambda} & \langle h_{\lambda}\rangle_q  \\ \hline \hline
() & 1 & 1\\
(4) & \frac{27}{4} \left(Q_2^2+2 Q_4\right) & \frac{9}{320}Q\\
(6) & \frac{225}{4} \left(63 Q_6+9 Q_2 Q_4+Q_2^3\right) & -\frac{55}{384}R\\
(3,3) & \frac{225}{4} \left(63 Q_3^2-108 Q_2 Q_4+2 Q_2^3\right) & \frac{115}{384}R \\
(8) & \frac{19845}{16} \left(3960 Q_8+360 Q_2 Q_6+20 Q_2^2 Q_4+Q_2^4\right) &  \frac{19173}{4096} Q^2\\
(5,3) & \frac{19845}{2} \left(495 Q_3Q_5 + 45 Q_2Q_3^2-1350 Q_2 Q_6 -50 Q_2^2Q_4 +2 Q_2^4\right) & -\frac{2415}{128} Q^2 \\
(4,4) & \frac{297675}{8} \left(132 Q_4^2 + 24 Q_2 Q_3^2-440Q_2 Q_6-28 Q_2^2Q_4 +Q_2^4\right) &-\frac{38241}{2048} Q^2 \\
(10) & \frac{382725}{8} \left(450450 Q_{10} + 30030 Q_2 Q_8 + 1155 Q_2^2 Q_6 + 35 Q_2^3 Q_4 + Q_2^5\right) & -\frac{2053485}{4096} QR \\
(7,3) & \frac{1913625}{8} \big(90090 Q_3 Q_7 + 6006 Q_2 Q_3 Q_5 - 336336 Q_2 Q_8 + 231 Q_2 Q_3^2 + \\ & \quad \quad \quad \quad - 12936 Q_2^2Q_6 -112 Q_2^3Q_4 + 10 Q_2^5\big) & \frac{11975985}{4096} QR \\
(6,4) & \frac{13395375}{8} \big( 12870 Q_4 Q_6 + 1716 Q_2 Q_3 Q_5 + 858 Q_2 Q_4^2 -96096 Q_2 Q_8 + \\ & \quad \quad \quad \quad + 132 Q_2^2Q_3^2 - 6501 Q_2^2 Q_6 -89 Q_2^3 Q_4+5 Q_2^4\big) & \frac{21255885}{4096} QR\\
(5,5) & \frac{8037225}{4} \big(10725 Q_5^2+1430 Q_2Q_3Q_5 + 1430 Q_2 Q_4^2-10010 Q_2 Q_8 + \\ & \quad \quad \quad \quad + 165 Q_2^2Q_3^2 - 7700 Q_2^2 Q_6 -120 Q_2^3Q_4+6 Q_2^5\big) & \frac{7759395}{1024} QR \\
(4,3,3) & \frac{13395375}{8} \big(12870 Q_3^2Q_4-34320 Q_2Q_3Q_5+10296Q_2 Q_4^2+363Q_2^2Q_3^2+ \\ & \quad \quad \quad \quad +55440Q_2^2 Q_6-376Q_2^3Q_4+10Q_2^5\big) & -\frac{16583805}{4096} QR \\
\hline
\end{array}
\]
In case $|\lambda|$ is odd the harmonic polynomials $h_{\lambda}$ up to weight $9$ are given in the following table. The $q$-bracket of odd degree (harmonic) polynomials is zero, hence trivially modular.
\[  \arraycolsep=5pt\def\arraystretch{1.5}
\begin{array}{l l}
\lambda & h_{\lambda} \\ \hline \hline
(3) 		& -\frac{9}{4}Q_3 				\\ 
(5) 		& -\frac{135}{4}\left(5Q_5+Q_2Q_3\right) 									\\ 
(7) & -\frac{14175}{16}\left(126Q_7+14Q_2Q_5+Q_2^2Q_3\right)  \\
(4,3) & -\frac{99225}{16}\left(18Q_3Q_4 - 40 Q_2 Q_5+Q_2^2Q_3\right)  \\ 
(9) & -\frac{297675}{8}\left(7722 Q_9 + 594Q_2 Q_7 +27 Q_2^2Q_5+Q_2^3Q_3\right)\\
(6,3) & -\frac{893025}{4}\left(1287 Q_3 Q_6 +99 Q_2Q_3Q_4-4158 Q_2Q_7-162 Q_2^2Q_5+5Q_2^3Q_3\right) \\
(5,4) &  - \frac{8037225}{8}\left(286 Q_4Q_5+66 Q_2Q_3Q_4-1540 Q_2Q_7 -117 Q_2^2Q_5+3Q_2^3Q_3\right) \\
(3,3,3) & -\frac{893025}{4}\left(1287 Q_3^3-3564Q_2Q_3Q_4+3240Q_2^2Q_5+10Q_2^3Q_3\right) \\
\hline
\end{array}
\]

\section*{Acknowledgment}
I would like to thank Gunther Cornelissen and Don Zagier for helpful discussions.


\end{document}